\documentclass[11pt,a4paper]{article}

\usepackage{amsmath,amssymb,graphicx,url,color,IEEEtrantools,tikz}

\usepackage{enumitem}
\setlist{itemsep=-0.5pt,topsep=.6ex,leftmargin=5ex} 

\evensidemargin=\oddsidemargin

\usepackage{amsthm}
\theoremstyle{plain}
\newtheorem{theorem}{Theorem}
\newtheorem{lemma}[theorem]{Lemma}

\newtheorem{corollary}[theorem]{Corollary}

\theoremstyle{definition}

\usepackage{etoolbox}
\AtEndEnvironment{example}{\qed}

\def\IEEEsubnumstart{\IEEEyesnumber\IEEEyessubnumber*}

\def\<{\langle}
\def\>{\rangle}
\def\...{,\ldots,}
\def\refeq#1{{(\ref{eq:#1})}}
\def\phi{\varphi}

\def\bb#1{\mathbb{#1}} 

\def\defmathop#1{\expandafter\newcommand\csname#1\endcsname{\mathop{\rm #1}}}
\defmathop{aff}
\defmathop{conv}
\defmathop{cone}
\defmathop{ri}
\defmathop{rb}
\defmathop{cl}
\defmathop{argmin}
\defmathop{epi}
\def\Span{\mathop{\rm span}}

\title{Relative Interior Rule in Block-Coordinate Minimization}

\author{Tom{\'a}{\v s} Werner, Daniel Pr{\r u}{\v s}a \\[1ex]
Department of Cybernetics, Faculty of Electrical Engineering, Czech Technical University}

\begin{document}

\maketitle

\section{Introduction}

(Block-)coordinate minimization is an iterative optimization method
which in every iteration finds a global minimum of the objective over
a variable or a subset of variables, while keeping the remaining
variables constant. For some problems, coordinate minimization converges
to a global minimum. This class includes unconstrained problems with
convex differentiable objective function \cite[\S2.7]{Bertsekas99} or
convex objective function whose non-differentiable part is
separable~\cite{Tseng:2001}. For general convex problems, the method
need not converge to a global minimum but only to a local one, where
`local' is meant with respect to moves along (subsets of) coordinates.

For large-scale non-differentiable convex problems, (block-)coordinate
minimization can be an acceptable option despite its inability to
converge to a global minimum.  An example is a class of methods to solve
the linear programming relaxation of the discrete energy minimization
problem (also known as MAP inference in graphical models).
These methods apply (block-)coordinate minimization to various forms
of the dual linear programming relaxation. Examples are max-sum
diffusion \cite{Kovalevsky-diffusion,Schlesinger-2011,Werner-PAMI07},
TRW-S~\cite{Kolmogorov06}, MPLP~\cite{Globerson08}, and
SRMP~\cite{Kolmogorov-PAMI-2015}. For many problems from computer
vision, it has been
observed~\cite{Szeliski:study:PAMI:2008,Kappes-study-2015} that TRW-S
converges faster than the competing methods and its fixed points are
often not far from global minima, especially for large sparse
instances.

When block-coordinate minimization is applied to a general convex
problem, in every iteration the minimizer over the current coordinate
block need not be unique and therefore a single minimizer must be
chosen. These choices can significantly affect the quality of the
achieved local minima. We propose that this minimizer should always be
chosen from the {\em relative interior\/} of the set of all minimizers
over the current block. Indeed, it can be easily verified that max-sum
diffusion satisfies this condition. We show that block-coordinate
minimization methods satisfying this condition are not worse, in a
certain precise sense, than any other block-coordinate minimization
methods.

\section{Main Results}
\label{sec:descent}

For brevity, we will use
\begin{equation}
M(X,f) 
= \{\, x\in X \mid f(x)\le f(y) \; \forall y\in X \,\}
\label{eq:M}
\end{equation}
to denote the set of all global minima of a function $f{:}\ Y\to\bb R$
on a set $X\subseteq Y$.

Suppose we want to
minimize a convex function $f{:}\ V\to\bb R$ on a closed convex set
$X\subseteq V$ where $V$~is a finite-dimensional vector space
over~$\bb R$. For that, we consider a coordinate-free generalization
of block-coordinate minimization. Let $\cal I$ be a finite set of
subspaces of~$V$, which represent search directions.
Having an estimate~$x_n$ of the minimum, the next estimate~$x_{n+1}$
is always chosen such that
\begin{equation}
x_{n+1} \in M(X\cap(x_n+I_n),f)
\label{eq:altopt}
\end{equation}
for some $I_n\in{\cal I}$.
Clearly, $f(x_{n+1})\le f(x_n)$.  A point $x\in X$ satisfying
\begin{equation}
x \in M(X\cap(x+I),f) \qquad\forall I\in{\cal I}
\label{eq:Ilocopt}
\end{equation}
has the property that $f$~cannot be improved by moving from~$x$
within~$X$ along any single subspace from~$\cal I$.  We call such a
point a {\em local minimum\/} of~$f$ on~$X$ with respect to~$\cal I$.
When $\cal I$ and/or $(X,f)$ is clear from context, we will speak only
about a local minimum of~$f$ on~$X$ or just a local minimum. Note that
the term `local minimum' is used here in a different meaning than is
usual in optimization and calculus.

Coordinate minimization and block-coordinate minimization are special
cases of this formulation. In the former, we have $V=\bb R^d$ and
${\cal I}=\{\Span\{e_1\}\...\Span\{e_d\}\}$ where $e_i$~denotes the
$i$th vector of the standard basis of~$\bb R^d$. In the latter,
we have $V=\bb R^d$ and each element of~$\cal I$ is the span of a subset of
the standard basis of~$\bb R^d$.

Recall \cite{Rockafellar-1970,Hiriart-book-2004} that the {\em
  relative interior\/} of a convex set $X\subseteq V$, denoted by
$\ri X$, is the topological interior of~$X$ with respect to the affine
hull of~$X$.  We propose to modify condition~\refeq{altopt} such that the
minimum is always chosen from the relative interior of the current
optimal set. Thus, \refeq{altopt}~changes to
\begin{equation}
x_{n+1} \in \ri M(X\cap(x_n+I_n),f) .
\label{eq:altopt-ri}
\end{equation}
A point~$x_{n+1}$ always exists because the relative interior of any
non-empty convex set is non-empty. We call a point $x\in X$ that
satisfies
\begin{equation}
x \in \ri M(X\cap(x+I),f) \qquad\forall I\in{\cal I}
\label{eq:Ilocopt-ri}
\end{equation}
an {\em interior local minimum\/} of~$f$ on~$X$ with respect
to~$\cal I$.  Clearly, every interior local minimum is a local
minimum.

In our analysis, another type of local minimum will naturally appear:
{\em pre-interior local minimum\/}.
It will be precisely defined later; informally, it is only a finite
number of iterations~\refeq{altopt-ri} away from an interior local
minimum.

Consider a sequence $(x_n)_{n\in\bb N}$ satisfying~\refeq{altopt}
resp.~\refeq{altopt-ri}, where $\bb N=\{1,2,\ldots\}$ denotes the
positive integers. To ensure that each search direction is always
visited again after a finite number of iterations, we assume that the
sequence $(I_n)_{n\in\bb N}$ contains each element of~$\cal I$ an
infinite number of times. For brevity, we will often write only
$(x_n)$ and $(I_n)$ instead of $(x_n)_{n\in\bb N}$ and
$(I_n)_{n\in\bb N}$. The following facts, proved in the sequel, show
that methods satisfying~\refeq{altopt-ri} are not worse, in a precise
sense, than methods satisfying~\refeq{altopt}:
\begin{itemize}
\item For every sequence $(x_n)$ satisfying~\refeq{altopt-ri}, if $x_1$~is an interior
local minimum then $x_n$~is an interior local minimum for all~$n$.
\item For every sequence $(x_n)$ satisfying~\refeq{altopt-ri}, if $x_1$~is a
pre-interior local minimum then $x_n$~is an
interior local minimum for some~$n$.
\item For every sequence $(x_n)$ satisfying~\refeq{altopt}, if $x_1$~is a pre-interior
local minimum then $f(x_n)=f(x_1)$ for all~$n$.
\item For every sequence $(x_n)$ satisfying~\refeq{altopt-ri}, if
$x_1$~is not a pre-interior local minimum then $f(x_n)<f(x_1)$ for some~$n$.
\end{itemize}

To illustrate this, consider an example of coordinate minimization
applied on a simple linear program (see the picture below). Let
$V=\bb R^2$, $X=\conv\{(1,0),(3,0),(3,1),(0,4)\}$,
$f(x)=\langle -e_1,x\rangle$ (i.e., $f$~is constant vertically and
decreases to the right), and ${\cal I}=\{\Span\{e_1\},\Span\{e_2\}\}$.
The set of global minima is the line segment $[(3,0),(3,1)]$, the set
of local minima is $[(3,0),(3,1)]\cup[(0,4),(3,1)]$, the set of
interior local minima is $\{(0,4)\}\cup\ri[(3,0),(3,1)]$, and the set
of pre-interior local minima is $\{(0,4)\}\cup[(3,0),(3,1)]$.  The
thick polyline shows the first few points of a sequence $(x_n)$
satisfying~\refeq{altopt-ri}, where the sequence $(I_n)$ alternates
between the two subspaces from~$\cal I$.  When starting from any point
$x_1\in X\setminus\{(0,4)\}$, every sequence $(x_n)$
satisfying~\refeq{altopt-ri} leaves any non-interior local minimum
after a finite number of iterations, while improving the objective
function. Informally, this is because when the objective cannot be
decreased by moving along any single subspace from~$\cal I$,
condition~\refeq{altopt-ri} at least enforces the point to move to a
face of~$X$ of a higher dimension (if such a face exists), providing
thus `more room' to hopefully decrease the objective in future
iterations. In contrast, condition~\refeq{altopt} allows a sequence
$(x_n)$ to stay in any (possibly non-interior) local minimum forever.
Of course, when starting from $x_1=(0,4)$, every sequence
satisfying~\refeq{altopt} will stay in~$x_1$ forever. This just
confirms the well-known fact that for some non-smooth convex problems,
coordinate minimization can get stuck in a point that is not a global
minimum.

\begin{center}
\includegraphics[scale=.4]{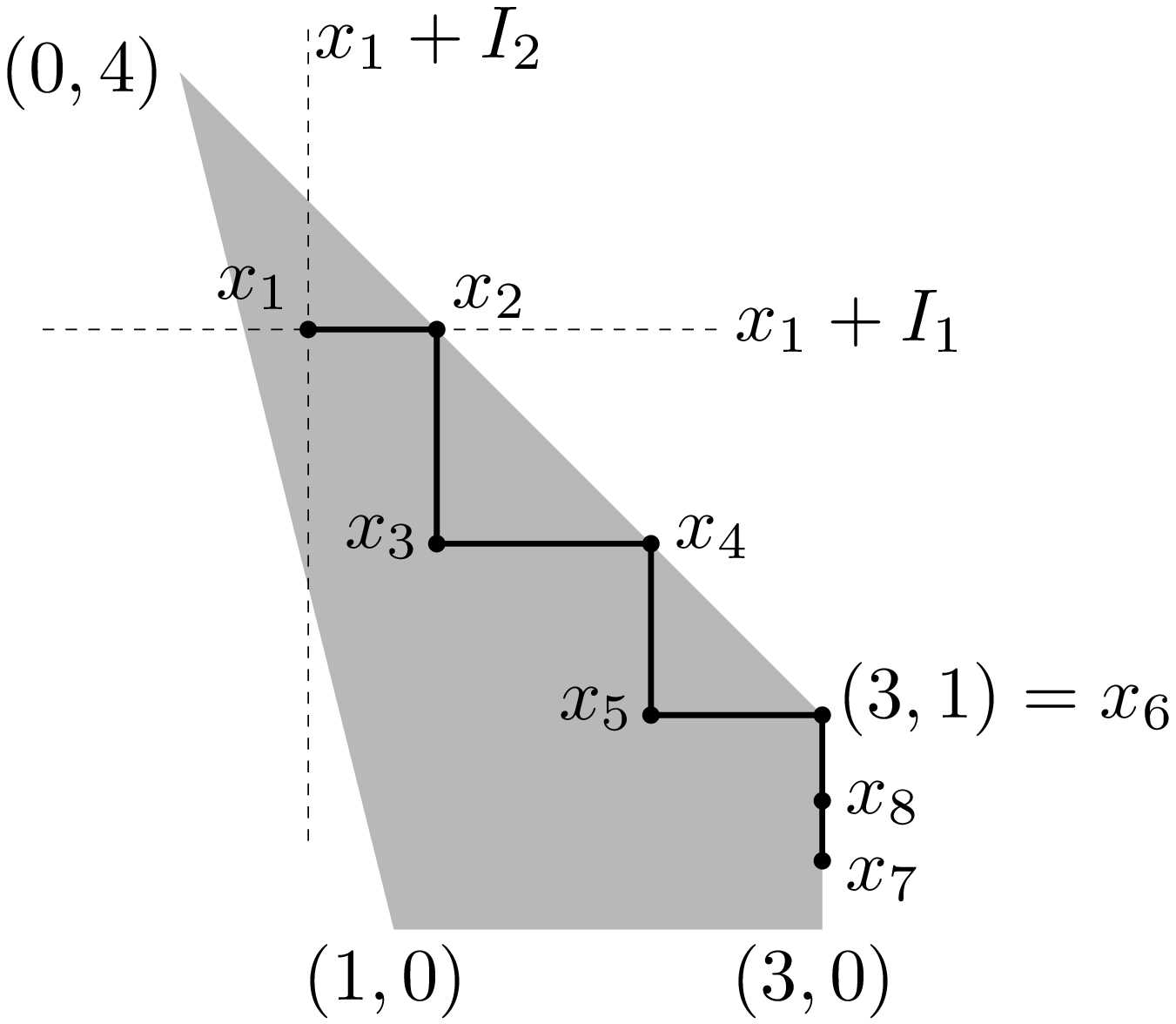}
\end{center}

Moreover, we prove the following convergence result: if the choices
in~\refeq{altopt-ri} are fixed such that $x_{n+1}$ is a continuous
function of~$x_n$, the elements of~$\cal I$ are visited in a cyclic
order, and the sequence $(x_n)$ is bounded, then the distances
of~$x_n$ from the set of pre-interior local minima converges to zero.

\section{Global Minima Are Local Minima}

As a warm-up, we prove one expected property of local minima: every
element of $M(X,f)$ (global minimum) is a local minimum and every
element of $\ri M(X,f)$ (which could be called {\em interior global
  minimum\/}) is an interior local minimum. Noting that global minima
are local minima with respect to~$\{V\}$, we actually prove, in
Theorem~\ref{th:domin} below, a more general fact.  For sets $\cal I$
and $\cal I'$ of subspaces of~$V$, we say that $\cal I'$ {\em
  dominates\/}~$\cal I$ if for every $I\in{\cal I}$ there is
$I'\in{\cal I'}$ such that $I\subseteq I'$.

\begin{lemma}
\label{th:MA}
Let $X,Y\subseteq V$ and $f{:}\ X\to\bb R$. Let
$M(X,f) \cap Y \neq \emptyset$. Then $M(X,f) \cap Y = M(X\cap Y,f)$.
\end{lemma}

\begin{proof}
To prove $\subseteq$, we need to prove that $x\in
M(X,f)\cap Y$ implies $x\in M(X\cap Y,f)$. This is obvious because if
$f(x)\le f(y)$ holds for all $y\in X$, then it holds for all $y\in
X\cap Y$.

To prove $\supseteq$, we need to prove that $x\in M(X\cap Y,f)$ and
$M(X,f)\cap Y\neq\emptyset$ imply $x\in M(X,f)$.
For that, it suffices to show that $x\in X\cap Y$ and $M(X,f)\cap
Y\neq\emptyset$ imply that $f(x)\le f(y)$ for all $y\in X\setminus
Y$. This is true, because $f(x)>f(y)$ for some $y\in X\setminus Y$
would imply $M(X,f)\cap Y=\emptyset$.
\end{proof}

Now we will use the property of the relative interior
\cite{Rockafellar-1970,Hiriart-book-2004} that for any convex sets
$X,Y\subseteq V$,
\begin{equation}
\ri X \cap \ri Y \neq \emptyset \quad\Longrightarrow\quad \ri X \cap \ri Y = \ri(X\cap Y) .
\label{eq:ri-cap}
\end{equation}

\begin{theorem}
\label{th:domin}
Let $X\subseteq V$ be a convex set and $f{:}\ X\to\bb R$ be a convex
function. Let $\cal I$ and $\cal I'$ be finite sets of subspaces
of~$V$ such that $\cal I'$ dominates $\cal I$.
\begin{itemize}
\item Every local minimum with respect to~$\cal I'$ is a local minimum
with respect to~$\cal I$.
\item Every interior local minimum with respect to~$\cal I'$ is an
interior local minimum with respect to~$\cal I$.
\end{itemize}
\end{theorem}

\begin{proof}
We just need to consider two subspaces $I,I'\subseteq V$ such that
$I\subseteq I'$.
\begin{itemize}
\item 
Noting that $x\in x+I$, by Lemma~\ref{th:MA} we have
$x\in M(X\cap(x+I'),f) = M(X\cap(x+I'),f)\cap(x+I)=M(X\cap(x+I),f)$.
\item 
Noting that $x+I=\ri(x+I)$, by~\refeq{ri-cap} we have
$x\in\ri M(X\cap(x+I'),f)=\ri M(X\cap(x+I'),f)\cap(x+I)=\ri(M(X\cap(x+I'),f)\cap(x+I))=\ri
M(X\cap(x+I),f)$.  \qedhere
\end{itemize}
\end{proof}

\section{Linear Objective Function}
\label{sec:altopt-lin}

Using the epigraph form, the minimization of a convex function on a
closed convex set can be transformed to the minimization of a linear
function on a closed convex set. Therefore, further
in~\S\ref{sec:altopt-lin} we assume that $X$~is closed convex and
$f$~is linear.
We will return to the case of non-linear convex~$f$ later
in~\S\ref{sec:epigraph}.

For $x,y\in V$, we denote
\begin{equation}
[x,y] = \conv\{x,y\} = \{\, (1-\alpha)x+y \mid 0\le\alpha\le1 \,\} .
\label{eq:segment}
\end{equation}
For $x\neq y$ this is a line segment, for $x=y$ it is a
singleton. It holds that
\begin{equation}
\ri[x,y] = \{\, (1-\alpha)x+y \mid 0<\alpha<1 \,\} .
\label{eq:segment-ri}
\end{equation}
For $x\neq y$ we have $\ri[x,y]=[x,y]\setminus\{x,y\}$, for $x=y$ we
have $[x,y]=\ri[x,y]=\{x\}$.


We recall basic facts about faces of a convex set
\cite{Rockafellar-1970,Hiriart-book-2004}.  A {\em face\/} of a convex
set $X\subseteq V$ is a convex set $F\subseteq X$ such that every line
segment from~$X$ whose relative interior intersects~$F$ lies in~$F$,
i.e.,
\begin{equation}
x,y\in X, \;\; F\cap\ri[x,y]\neq\emptyset \quad\Longrightarrow\quad x,y\in F .
\label{eq:face}
\end{equation}
The set of all faces of a closed convex set partially ordered by
inclusion is a complete lattice, in particular it is closed under
(possibly infinite) intersections.
For a point $x\in X$, let $F(X,x)$ denote the intersection of all
faces (equivalently, the smallest face) of~$X$ that contain~$x$. For
every $x,y\in X$,
\begin{IEEEeqnarray}{rCrClCl}
\label{eq:Fx}
y&\in& F(X,x) &\quad\Longleftrightarrow\quad& F(X,y)&\subseteq& F(X,x) , \IEEEsubnumstart\label{eq:Fx<=}\\
y&\in& \ri F(X,x) &\quad\Longleftrightarrow\quad& F(X,y)&=&F(X,x) . \label{eq:Fx=}\\
y&\in& \rb F(X,x) &\quad\Longleftrightarrow\quad& F(X,y)&\subsetneq& F(X,x) , \label{eq:Fx<}
\end{IEEEeqnarray}
where $\rb X=X\setminus\ri X$ denotes the relative boundary of a
closed convex set~$X$. Equivalence~\refeq{Fx=} shows that $F(X,x)$ is
in fact the unique face of~$X$ having~$x$ in its relative interior.
Note that \refeq{Fx<} follows from \refeq{Fx<=} and~\refeq{Fx=}.

The following simple lemmas will be used several times later:

\begin{lemma}
\label{th:ri-prolong}
Let $X\subseteq V$ be a convex set. We have $x\in\ri X$ iff
for every $y\in X$ there exists $u\in X$ such that $x\in\ri[y,u]$.
\end{lemma}

\begin{proof}
The `only-if' direction is immediate from the definition of relative
interior. For the `if' direction see, e.g.,
\cite[Theorem~6.4]{Rockafellar-1970}.
\end{proof}

\begin{lemma}
\label{th:XYri}
Let $X,Y\subseteq V$ be closed convex sets such that $Y\subseteq X$. Let $x\in\ri Y$. Then
\begin{IEEEeqnarray}{rCrCrCr}
\label{eq:XYri}
y &\in& Y &\quad\Longrightarrow\quad& y &\in& F(X,x) \IEEEsubnumstart\label{eq:XYri:<=}\\
y &\in& \ri Y &\quad\Longrightarrow\quad& y &\in& \ri F(X,x) \label{eq:XYri:=}\\
y &\in& \rb Y &\quad\Longrightarrow\quad& y &\in& \rb F(X,x) \label{eq:XYri:<}
\end{IEEEeqnarray}
\end{lemma}

\begin{proof}
To see~\refeq{XYri:<=}, let $x\in\ri Y$ and $y\in Y$. Thus, by
Lemma~\ref{th:ri-prolong}, there is $u\in Y$ such that
$x\in\ri[u,y]$. Since $x\in F(X,x)$ and $y,u\in X$, the definition of
face yields $y\in F(X,x)$. Implications \refeq{XYri:=}
and~\refeq{XYri:<} follow from~\refeq{XYri:<=} and~\refeq{Fx}.
\end{proof}

\begin{lemma}
\label{th:xyzu}
Let $y,z,u\in V$ and $x\in\ri[u,y]$. Then $\ri[u,z]\cap\ri[x,x+z-y]\neq\emptyset$.
\end{lemma}

\begin{proof}
Let $0<\alpha<1$ be such that $x=(1-\alpha)u+\alpha y$ (note that if
$y\neq u$ then $\alpha$ is unique, otherwise we can choose any
$0<\alpha<1$). Let $v=(1-\alpha)u+\alpha z$, hence
$v\in\ri[u,z]$. Subtracting the two equations yields
$v=(1-\alpha)x+\alpha(x+z-y)$, hence $v\in\ri[x,x+z-y]$.
\end{proof}

The picture illustrates Lemma~\ref{th:xyzu} for the points in a general
position (i.e., $y,z,u$ not collinear):
\begin{center}
\begin{tikzpicture}
\def\ptsize{1.5pt}
\coordinate (beg) at (0,0);
\coordinate (end) at (4,0);
\coordinate (y) at (.5,0);
\coordinate (x) at (2.5,0);
\coordinate (u) at (3.5,0);
\coordinate (z) at (1,1.5);
\coordinate (xx) at (3,1.5);
\coordinate (v) at (intersection of z--u and x--xx);

\fill (y) circle (\ptsize) node[below] {$y$};
\fill (x) circle (\ptsize) node[below] {$x$};
\fill (u) circle (\ptsize) node[below] {$u$};
\fill (z) circle (\ptsize) node[left] {$z$};
\fill (v) circle (\ptsize) node[above=2pt,right] {$v$};
\fill (xx) circle (\ptsize) node[right] {$x+z-y$};

\draw (beg) -- (end);
\draw[dashed] (y) -- (z);
\draw (u) -- (z);
\draw[dashed] (x) -- (xx);
\end{tikzpicture}
\end{center}

\subsection{Structure of the Set of Local Minima}
\label{sec:structure}

It is well-known that the set of global minima of a linear
function~$f$ on a closed convex set~$X$ is an (exposed) face
of~$X$. We show that local resp.\ interior local minima also cluster
to faces of~$X$. Moreover, similarly as the set of all faces of~$X$,
we show that the set of faces of~$X$ containing local resp.\ interior
local minima are closed under intersections.

In the theorems in the rest of this section, the letter~$I$ will always denote
a subspace of~$V$.

\begin{theorem}
\label{th:xF-locmin}
Let $x\in M(X \cap (x+I),f)$ and $y\in F(X,x)$. Then
$y\in M(X \cap (y+I),f)$.
\end{theorem}

\begin{proof}
Let $z\in X \cap (y+I)$. We need to prove that $f(y)\le f(z)$.
Since $y\in F(X,x)$, by Lemma~\ref{th:ri-prolong} there is $u\in X$
such that $x\in\ri[u,y]$. By Lemma~\ref{th:xyzu}, there is a point
\[
v\in\ri[u,z]\cap\ri[x,x+z-y].
\]
Since $z,u\in X$, from convexity of~$X$ we have $v\in X$. Since
$z-y\in I$, we have $v\in x+I$.  Since $x\in M(X \cap (x+I),f)$, we
thus have $f(x)\le f(v)$, hence $f(x)\le f(x+z-y)$. Since
$[x,x+z-y]=[y,z]+x-y$, by linearity of~$f$ we have $f(y)\le f(z)$.
\end{proof}

\begin{corollary}
\label{th:xF-locmin'}
If $x$ is a local minimum, then every point of $F(X,x)$ is a local
minimum.
\end{corollary}

But notice that if $x$ and~$y$ are local minima such that $y\in F(X,x)$,
then we can have $f(y)\neq f(x)$.

\begin{lemma}
\label{th:xy-MF}
Let $x\in\ri M(X\cap(x+I),f)$ and $y\in F(X,x)$. Then $M(X\cap(y+I),f) \subseteq F(X,x)$.
\end{lemma}

\begin{proof}
Let $z\in M(X\cap(y+I),f)$.  By Theorem~\ref{th:xF-locmin} we have
$y\in M(X\cap(y+I),f)$, hence $f(z)=f(y)$.  Since $y\in F(X,x)$, by
Lemma~\ref{th:ri-prolong} there is $u\in X$ such that $x\in\ri[u,y]$.
By Lemma~\ref{th:xyzu}, there is
\[
v\in\ri[u,z]\cap\ri[x,x+z-y] .
\]
Since $z,u\in X$ and $z-y\in I$, we have $v\in X\cap(x+I)$. Since
$[x,x+z-y]=[y,z]+x-y$, by linearity of~$f$ we have $f(v)=f(x)$, hence
$v\in M(X\cap(x+I),f)$.  Lemma~\ref{th:XYri} yields $v\in
F(X,x)$. Since $z,u\in X$, the definition of face yields
$z\in F(X,x)$.
\end{proof}

\begin{lemma}
\label{th:MF-xy}
Let $x\in M(X\cap(x+I),f) \subseteq F(X,x)$. Then $x\in\ri M(X\cap(x+I),f)$.
\end{lemma}

\begin{proof}
Let $u\in M(X\cap(x+I),f)$. Therefore $f(u)=f(x)$. Moreover, by
Lemma~\ref{th:ri-prolong} there is $v\in F(X,x)$ such that
$x\in\ri[u,v]$. Since $u\in x+I$, we have $v\in x+I$. By linearity
of~$f$ we have $f(v)=f(x)$, therefore $v\in M(X\cap(x+I),f)$. By
Lemma~\ref{th:ri-prolong}, $x\in\ri M(X\cap(x+I),f)$.
\end{proof}

\begin{theorem}
\label{th:meetF-ri}
Let $Y\subseteq X$. Let $x\in\ri M(X\cap(x+I),f)$ for all $x\in
Y$. Let $y\in\ri\bigcap_{x\in Y}F(X,x)$. Then
$y\in\ri M(X\cap(y+I),f)$.
\end{theorem}

\begin{proof}
Since $G=\bigcap_{x\in Y} F(X,x)$ is a face of~$X$, we have
$y\in\ri G$ iff $G=F(X,y)$. By Theorem~\ref{th:xF-locmin},
$y\in M(X\cap(y+I),f)$. By Lemma~\ref{th:xy-MF},
$M(X\cap(y+I),f)\subseteq G$. By Lemma~\ref{th:MF-xy},
$y\in\ri M(X\cap(y+I),f)$.
\end{proof}

\begin{corollary}
\label{th:meetF-ri'}
Let $Y\subseteq X$. If every point from~$Y$ is an interior local minimum, then
every relative interior point of the face $\bigcap_{x\in Y} F(X,x)$ is an interior local minimum.
\end{corollary}

\begin{corollary}
\label{th:xF-locmin-ri'}
If $x$~is an interior local minimum, then every point of
$\ri F(X,x)$ is an interior local minimum.
\end{corollary}

\begin{proof}
This is Corollary~\ref{th:meetF-ri'} for $Y=\{x\}$.
\end{proof}

The results from this section lead to the following definitions and
facts:
\begin{itemize}
\item We call a face of~$X$ a {\em local minima face\/} if all its
points are local minima. Since the set of faces of~$X$ is closed under
intersection, it follows from Corollary~\ref{th:xF-locmin'} that the set
of all local minima faces of~$X$ (assuming fixed~$f$ and~$\cal I$) is
closed under intersections. Thus, it is a complete meet-semilattice
(but not a lattice, because it need not have the greatest element).
\item We call a face of~$X$ an {\em interior local minima face\/} if
all its relative interior points are interior local
minima. Corollary~\ref{th:meetF-ri'} shows that the set of all interior
local minima faces of~$X$ (assuming fixed~$f$ and~$\cal I$) is closed
under intersections. Thus, it again is a complete meet-semilattice.
\end{itemize}

We finally define one more type of local minimum: a point~$x$ is a
{\em pre-interior local minimum\/} if $x\in F(X,y)$ for some interior
local minimum~$y$. Motivation for introducing this concept will become
clear later.

\subsection{The Effect of Iterations}
\label{sec:iterations}

Here we prove properties of sequences $(x_n)$ satisfying
conditions~\refeq{altopt} resp.~\refeq{altopt-ri} under various
assumptions.

\begin{theorem}
\label{th:n-iters-ri}
Let $(x_n)$ be a sequence satisfying~\refeq{altopt-ri} such that
$x_1$~is an interior local minimum. Then for all~$n$ we have
$f(x_{n+1})=f(x_n)$, $x_{n+1}\in\ri F(X,x_n)$, and $x_n$~is an interior
local minimum.
\end{theorem}

\begin{proof}
Suppose that for some~$n$, $x_n$~is an interior local
minimum. Considering~\refeq{altopt-ri}, by Lemma~\ref{th:XYri} we thus
have $x_{n+1}\in\ri F(X,x_n)$. By Corollary~\ref{th:meetF-ri'},
$x_{n+1}$~is an interior local minimum. Since
$x_n,x_{n+1}\in\ri M(X\cap(x_n+I_n),f)$, we have $f(x_{n+1})=f(x_n)$.
\end{proof}

\begin{theorem}
\label{th:iter-constant}
Let $(x_n)$ be a sequence satisfying~\refeq{altopt-ri} and
$f(x_{n+1})=f(x_n)$ for all~$n$. Then for all~$n$ we have
$x_n\in F(X,x_{n+1})$, there exists~$n$ such that $x_n$~is an interior
local minimum, and $x_1$~is a pre-interior local minimum.
\end{theorem}

\begin{proof}
Combining $f(x_{n+1})=f(x_n)$ with~\refeq{altopt-ri} yields
$x_n\in M(X\cap(x_n+I_n),f)$. Thus, for every~$n$ there are two
possibilities:
\begin{itemize}
\item If $x_n\in\ri M(X\cap(x_n+I_n),f)$ then, by Lemma~\ref{th:XYri},
we have $x_n\in\ri F(X,x_{n+1})$. By Theorem~\ref{th:meetF-ri}, we
have $x_{n+1}\in\ri M(X\cap(x_{n+1}+I),f)$ for all $I\in\cal I$ such
that $x_n\in\ri M(X\cap(x_n+I),f)$.
\item If $x_n\in\rb M(X\cap(x_n+I_n),f)$ then, by Lemma~\ref{th:XYri}, we have
$x_n\in\rb F(X,x_{n+1})$.
\end{itemize}
In either case, we have $x_n\in F(X,x_{n+1})$.  Moreover, if $x_n$~is
not an interior local minimum for some~$n$, then after some finite
number~$m$ of iterations the second case occurs, therefore
$x_n\in\rb F(X,x_{n+m})$. But this implies $\dim F(X,x_{n+m})>\dim F(X,x_n)$.
If $x_n$~were not an interior local minimum for any~$n$, for some~$n$
we would have $\dim F(X,x_n)>\dim X$, which is impossible.

Since $x_n\in F(X,x_{n+1})$ for all~$n$, the faces
$F(X,x_1)\subseteq F(X,x_2)\subseteq\cdots$ form a non-decreasing
chain. In particular, $x_1\in F(X,x_n)$ for all~$n$. Since there
is~$n$ such that $x_n$~is an interior local minimum, $x_1$~is a
pre-interior local minimum.
\end{proof}

\begin{theorem}
\label{th:captured}
Let $(x_n)$ be a sequence satisfying~\refeq{altopt} such that $x_1$~is
a pre-interior local minimum, i.e., $x_1\in F(X,x)$ for some interior
local minimum~$x$. Then for all~$n$ we have $x_n\in F(X,x)$ and
$f(x_n)=f(x_1)$.
\end{theorem}

\begin{proof}
We will use induction on~$n$. The claim trivially holds for $n=1$. We
will show that for every~$n$, $x_n\in F(X,x)$ implies $x_{n+1}\in F(X,x)$ and
$f(x_{n+1})=f(x_n)$.

Let $x_n\in F(X,x)$. By Lemma~\ref{th:ri-prolong}, there is $u\in X$
such that $x\in\ri[x_n,u]$.  By Lemma~\ref{th:xyzu}, there is
\[
v\in\ri[u,x_{n+1}]\cap\ri[x,x+x_{n+1}-x_n].
\]
Since $u,x_{n+1}\in X$, we have $v\in X$.  Since $x_{n+1}-x_n\in I_n$,
we have $v\in x+I_n$. Since $x\in M(X\cap(x+I_n),f)$, this implies
$f(x)\le f(v)$.  Since $[x,x+x_{n+1}-x_n]=[x_n,x_{n+1}]+x-x_n$, by
linearity of~$f$ we have $f(x_n)\le f(x_{n+1})$. But
from~\refeq{altopt} we have also $f(x_{n+1})\le f(x_n)$, hence
$f(x_{n+1})=f(x_n)$. This in turn implies $f(v)=f(x)$.  Since
$x\in\ri M(X \cap (x+I_n),f)$, we have $v \in M(X \cap
(x+I_n),f)$. By Lemma~\ref{th:XYri}, $v\in F(X,x)$.  Since
$u,x_{n+1}\in X$ and $v\in F(X,x)$, the definition of face gives $x_{n+1}\in F(X,x)$.
\end{proof}

\begin{corollary}
Let $(x_n)$ be a sequence satisfying~\refeq{altopt-ri} such that
$x_1$~is a pre-interior local minimum. Then there exists~$n$ such that
$x_n$~is an interior local minimum.
\end{corollary}

\begin{proof}
First apply Theorem~\ref{th:captured} and then Theorem~\ref{th:iter-constant}.
\end{proof}

\begin{corollary}
Let $(x_n)$ be a sequence satisfying~\refeq{altopt-ri}. Then $x_1$~is
a pre-interior local minimum iff $f(x_n)=f(x_1)$ for all~$n$.
\end{corollary}

\begin{proof}
The `if' direction follows from Theorem~\ref{th:iter-constant}.
The `only-if' direction follows from Theorem~\ref{th:captured}.
\end{proof}

\subsection{Convergence}
\label{sec:limit-point}

So far, we have not examined the convergence properties of sequences
$(x_n)$ satisfying~\refeq{altopt-ri}. For that, we impose some
additional restrictions on the sequences $(x_n)$ and $(I_n)$. Namely,
we assume that the action of every iteration is continuous and the
elements of~$\cal I$ are visited in a regular order.

Formally, we assume that for each $I\in\cal I$ a continuous map
$p_I{:}\ X\to X$ is given that satisfies
\begin{equation}
p_I(x)\in\ri M(X\cap(x+I),f)
\label{eq:pI}
\end{equation}
for every $x\in X$. This map describes the action of one
iteration. Let the map
\begin{equation}
p_\sigma = p_{\sigma(1)} \circ \cdots \circ p_{\sigma(m)}
\label{eq:psigma}
\end{equation}
denote the action of one round of iterations, in which all elements
of~$\cal I$ are visited (some possibly more than once) in the order
given by a surjective map $\sigma{:}\ \{1\...m\}\to{\cal I}$ where
$m\ge{\cal I}$.

In Theorem~\ref{th:iter-constant}, the sequence $(I_n)$ is assumed to
contain every element of~$\cal I$ an infinite number of times. The
form of iterations given by~$p_\sigma$ gives a stronger property: each
element of~$\cal I$ is always visited again after at most
$m$~iterations. We adapt Theorem~\ref{th:iter-constant} to this
situation. For that, we denote $p=p_\sigma^{d+1}$ (i.e., $p$~is
obtained by composing~$p_\sigma$ with itself $(d+1)$-times) where
$d=\dim X$.

\begin{theorem}
\label{th:cycle}
Let $x\in X$ and $f(p(x))=f(x)$. Then $p(x)$ is an interior
local minimum and $x$~is a pre-interior local minimum.
\end{theorem}

\begin{proof}
By similar arguments as in the proof of
Theorem~\ref{th:iter-constant}, for every $x\in X$ it holds that:
\begin{itemize}
\item If $x$~is an interior local minimum, then $x\in\ri F(X,p_\sigma(x))$.
\item If $x$~is not an interior local minimum, then
$x\in\rb F(X,p_\sigma(x))$, hence $\dim F(X,p_\sigma(x))>\dim F(X,x)$.
\end{itemize}
Therefore, if $f(p(x))=f(x)$ and $p(x)$ were not an interior local minimum,
we would have $\dim F(X,p(x))>\dim X$, a contradiction. Since
$x\in F(X,p(x))$, $x$~is a pre-interior local minimum.
\end{proof}

Starting from some $x\in X$, we will examine convergence properties of
the sequence $(x_n)$ defined by $x_n=p^n(x)$. Recall that a {\em
  limit point\/} (also known as an accumulation point or cluster
point) of a sequence is the limit point of its converging subsequence.

\begin{theorem}
\label{th:accum}
Let $x\in X$. Let the sequence $(f(p^n(x)))_{n\in\bb N}$ be bounded.
Then every limit point $y$ of the sequence $(p^n(x))_{n\in\bb N}$
satisfies $f(p(y))=f(y)$.
\end{theorem}

\begin{proof}
Let us denote $x_n=p^n(x)$. Let $y$~be a limit point of the
sequence $(x_n)$, i.e., for some strictly increasing function
$k{:}\ \bb N\to\bb N$ we have
\begin{equation}
\lim_{n\to\infty} x_{k(n)} = y .
\label{eq:accum-proof:1}
\end{equation}
Since $p$~is a composition of a finite number of continuous maps, it
is continuous. Applying~$p$ to~\refeq{accum-proof:1} yields
\begin{equation}
p\bigl(\, \lim_{n\to\infty} x_{k(n)} \bigr) = 
\lim_{n\to\infty} p(x_{k(n)})
= \lim_{n\to\infty} x_{k(n)+1}
= p(y) .
\label{eq:accum-proof:2}
\end{equation}
We show that
\begin{equation}
f(y)
= \lim_{n\to\infty} f(x_{k(n)})
= \lim_{n\to\infty} f(x_n)
= \lim_{n\to\infty} f(x_{k(n)+1})
= f(p(y)) .
\end{equation}
The first and last equality holds by applying the continuous
function~$f$ to equality~\refeq{accum-proof:1}
and~\refeq{accum-proof:2}, respectively. The second and third equality
hold because the sequence $(f(x_n))$ is convergent (being bounded
and non-increasing), hence every its subsequence converges to the same
point.
\end{proof}

\begin{corollary}
\label{th:accum'}
Let $x\in X$. Let the sequence $(f(p^n(x)))_{n\in\bb N}$ be bounded.
Then every limit point $y$ of the sequence
$(p^n(x))_{n\in\bb N}$ is a pre-interior local minimum.
\end{corollary}

\begin{proof}
Combine Theorems~\ref{th:accum} and~\ref{th:cycle}.
\end{proof}

Let $d{:}\ X\times X\to\bb R_+$ be a metric on~$X$. Denote the
distance of a point $x\in V$ from a set $X\subseteq V$ as
\begin{equation}
d(X,x) = \inf_{y\in X} d(x,y) .
\label{eq:d}
\end{equation}


\begin{lemma}
\label{th:d-continuous}
For any $X\subseteq V$, the function $x\mapsto d(X,x)$ is Lipschitz, hence continuous.
\end{lemma}

\begin{proof}
For all $x,y\in V$ and $z\in X$ we have
$d(X,x) \le d(x,z) \le d(x,y) + d(y,z)$.
Taking $\inf$ over~$z$ on the right gives
$d(X,x) \le d(x,y) + d(X,y)$. Swapping $x$ and~$y$ 
gives $|d(X,x)-d(X,y)| \le d(x,y)$.
\end{proof}

\begin{lemma}
\label{th:fY-bounded}
Let $X\subseteq V$ be closed, $Y\subseteq X$ bounded, and 
$f{:}\ X\to\bb R$ continuous. Then $f(Y)$ is bounded.
\end{lemma}

\begin{proof}
By monotonicity of closure, $\cl Y\subseteq \cl X=X$. The set $\cl Y$
is compact (closed and bounded), therefore $f(\cl Y)$ is also
compact. Hence $f(Y)\subseteq f(\cl Y)$ is bounded.
\end{proof}

\begin{lemma}
\label{th:subseq-conv}
A sequence in a metric space is convergent iff it is bounded and has a unique limit point.
\end{lemma}

\begin{proof}
The `only-if' direction is obvious. To see the `if' direction, let
$x$~be a limit point of a bounded sequence $(x_n)$. For
contradiction, suppose $(x_n)$ does not converge to~$x$.  Then for
some $\epsilon>0$, for every~$n_0$ there is $n>n_0$ such that
$d(x_n,x)>\epsilon$.  So there is a subsequence $(y_n)$ such that
$d(y_n,x)>\epsilon$ for all~$k$.  As $(y_n)$~is bounded, by
Bolzano-Weierstrass it has a convergent subsequence, $(z_n)$. But
$(z_n)$ clearly cannot converge to~$x$.
\end{proof}

\begin{theorem}
\label{th:d-converges}
Let $(x_n)_{n\in\bb N}$ be a bounded sequence from a closed set
$X\subseteq V$. Let $Y\subseteq X$ be such that every limit point of
$(x_n)$ is in~$Y$. Then $\lim\limits_{n\to\infty}d(Y,x_n)=0$.
\end{theorem}

\begin{proof}
By Lemmas~\ref{th:d-continuous} and~\ref{th:fY-bounded}, the sequence
$(d(Y,x_n))$ is bounded.
Thus it has a convergent subsequence, $(d(Y,y_n))$ where $(y_n)$ is a
subsequence of $(x_n)$. By Lemma~\ref{th:subseq-conv}, it
suffices to show that $\lim\limits_{n\to\infty}d(Y,y_n)=0$.

Being a subsequence of $(x_n)$, the sequence $(y_n)$ is
bounded. Therefore, it has a convergent subsequence, $(z_n)$. Thus,
$x=\lim\limits_{n\to\infty}z_n$ is a limit point of $(x_n)$. Therefore, $d(Y,x)=0$.
Applying the continuous function $x\mapsto d(Y,x)$ to this limit
yields $0=d(Y,x)=\lim\limits_{n\to\infty}d(Y,z_n)$.  Since the sequence
$(d(Y,y_n))$ is convergent, every its convergent subsequence converges
to the same number. Since $(d(Y,z_n))$ is one such subsequence, we have
$\lim\limits_{n\to\infty}d(Y,y_n)=\lim\limits_{n\to\infty}d(Y,z_n)=0$.
\end{proof}

\begin{corollary}
\label{th:e-converges}
Let $x\in X$. Let the sequence $(p^n(x))_{n\in\bb N}$ be bounded. Let
$Y$~be the set of all pre-interior local minima of~$f$ on~$X$. Then
$\lim\limits_{n\to\infty} d(Y,p^n(x))=0$.
\end{corollary}

\begin{proof}
Combine Theorem~\ref{th:d-converges} and Corollary~\ref{th:accum'}.
\end{proof}

For the sequence $(p^n(x))_{n\in\bb N}$ to be bounded, it clearly
suffices that $X$~is bounded. But there is a weaker sufficient
condition: as the sequence $(f(p^n(x)))_{n\in\bb N}$ is
non-increasing, it suffices that the set
$X\cap\{\,y\in V\mid f(y)\le f(x)\,\}$ is bounded (note that
$\{\,y\in V\mid f(y)\le f(x)\,\}$ is the half-space whose boundary is
the contour of~$f$ passing through the initial point~$x$).

\section{Non-linear Objective Function}
\label{sec:epigraph}

As we said, the minimization of a convex function on a convex set can
be transformed to the epigraph form, which is the minimization of a
linear function on a convex set. Here we show that this transformation
allows us to generalize the results from~\S\ref{sec:altopt-lin} to
non-linear convex objective functions.

The {\em epigraph\/} of a function $f{:}\ X\to\bb R$ is the set
\begin{equation}
\epi f = \{\,(x,t)\in X\times\bb R\mid f(x)\le t\,\} .
\label{eq:epigraph}
\end{equation}
If $X\subseteq V$ is closed convex and $f$~is convex, then $\epi f$ is
closed convex. 
We have
\begin{equation}
\min_{x\in X}f(x) \;=\; \min_{(x,t)\in\epi f} t \;=\; \min_{\bar x\in\epi f}\pi(\bar x)
\label{eq:epiform-orig}
\end{equation}
where $\pi{:}\ V\times\bb R\to\bb R$ is the linear function defined by
$\pi(x,t)=t$, i.e., the projection on the $t$-coordinate. For every
$(x,t)\in M(\epi f,\pi)$ we have $t=f(x)$, i.e., $t$~is the minimum
value of~$f$ on~$X$. Moreover,
\begin{IEEEeqnarray}{rCr}
\label{eq:epiform'}
M(X,f) \times \{t\} &=& M(\epi f,\pi) , \IEEEsubnumstart\\
\ri M(X,f) \times \{t\} &=& \ri M(\epi f,\pi) ,
\end{IEEEeqnarray}
which can equivalently be written as
\begin{IEEEeqnarray}{rCrCrCr}
\label{eq:epiform}
x &\in& M(X,f) &\quad\Longleftrightarrow\quad& (x,f(x)) &\in& M(\epi f,\pi) , \IEEEsubnumstart\\
x &\in& \ri M(X,f) &\quad\Longleftrightarrow\quad& (x,f(x)) &\in& \ri M(\epi f,\pi) .
\end{IEEEeqnarray}

The following lemma will allow us to show that the concepts of local
minima and the updates~\refeq{altopt} and~\refeq{altopt-ri} remain
`the same' if we pass to the epigraph form, provided that instead of a
subspace~$I$ we use the subspace $\bar I=I\times\bb R$.
To illustrate this, consider the case $X=V=\bb R^d$ and
coordinate minimization. In every iteration, we minimize $f(x_1\...x_d)$
over a single variable~$x_i$. In the epigraph form, we would
minimize~$t$ subject to $f(x_1\...x_d)\le t$ over the pair
$(x_i,t)$. Clearly, both forms are equivalent.

\begin{lemma}
\label{th:epiform-Y}
Let $X\subseteq V$ be convex, $f{:}\ X\to\bb R$ be convex. Let
$I\subseteq V$ be a subspace and $\bar I=I\times\bb R$. Let
$\bar x=(x,t)\in \epi f$ and $y\in X$. Then
\begin{IEEEeqnarray}{rCrCrCr}
\label{eq:epiform-Y}
y &\in& M(X\cap(x+I),f) &\quad\Longleftrightarrow\quad& (y,f(y)) &\in& M(\epi f\cap(\bar x+\bar I),\pi) , \IEEEsubnumstart\\
y &\in& \ri M(X\cap(x+I),f) &\quad\Longleftrightarrow\quad& (y,f(y)) &\in& \ri M(\epi f\cap(\bar x+\bar I),\pi) .
\end{IEEEeqnarray}
\end{lemma}

\begin{proof}
One can verify from~\refeq{epigraph} that for every $Y\subseteq V$ we have
\begin{align*}
\epi f\cap(Y\times\bb R)
&= \epi f|_{X\cap Y}
\end{align*}
where $f|_{X\cap Y}$ denotes the restriction of the function~$f$ to
the set $X\cap Y$. Since
$\bar x+\bar I = (x,t')+(I\times\bb R) = (x+I)\times(t'+\bb R) =
(x+I)\times\bb R$, we thus have
$\epi f\cap(\bar x+\bar I) = \epi f|_{X\cap(x+I)}$.  We see that
\refeq{epiform-Y} are \refeq{epiform}, applied to the function
$f|_{X\cap(x+I)}$.
\end{proof}

By letting $y=x$ and $t=f(x)$, the lemma shows that $x$~is an
[interior] local minimum of~$f$ on~$X$ with respect to~$\cal I$ iff
$(x,f(x))$ is an [interior] local minimum of~$\pi$ on $\epi f$ with
respect to $\bar{\cal I}=\{\,I\times\bb R\mid I\in{\cal I}\,\}$.
Similarly, the results from~\S\ref{sec:iterations}
and~\S\ref{sec:limit-point} can be extended from linear to non-linear
convex functions~$f$.

\end{document}